\documentclass[a4paper,11pt]{article}
\usepackage{braket,hyperref}
\usepackage{amssymb,mathtools,amsthm,amsmath,bm}
\usepackage{tikz}

\usepackage{xargs}
\usepackage{gensymb}

\usepackage{enumitem}
\usepackage{verbatim}

\usepackage{kbordermatrix}%

\renewcommand{\kbldelim}{(}
\renewcommand{\kbrdelim}{)}
\usepackage{thm-restate}

\theoremstyle{plain}
\newtheorem{theorem}{\bf Theorem}[section]
\newtheorem{claim}[theorem]{\bf Claim}
\newtheorem{remark}[theorem]{\bf Remark}
\newtheorem{problem}[theorem]{\bf Problem}
\newtheorem{conjecture}[theorem]{\bf Conjecture}
\newtheorem{proposition}[theorem]{\bf Proposition}
\newtheorem{corollary}[theorem]{\bf Corollary}
\newtheorem{lemma}[theorem]{\bf Lemma}
\theoremstyle{definition}
\newtheorem{definition}[theorem]{\bf Definition}

\newcommand{\Rea}{{\mathbb R}}

\DeclareMathOperator{\rank}{\text{rank}}

\newcommand{\lfrac}[2]{\left\lfloor\frac{#1}{#2}\right\rfloor}

\newcommand{\ufrac}[2]{\left\lceil\frac{#1}{#2}\right\rceil}

\newcommand{\bv}{\mathbf v}

\newcommand{\px}{p^*}
\usepackage{authblk}

	\title{Rigidity expander graphs}

\author[2]{Alan Lew\thanks{\href{mailto:alanlew@andrew.cmu.edu}{alanlew@andrew.cmu.edu}. Alan Lew was partially supported by the Israel Science Foundation grant ISF-2480/20.}}
\author[1]{Eran Nevo\thanks{\href{mailto:nevo@math.huji.ac.il}{nevo@math.huji.ac.il}. Eran Nevo was partially supported by the Israel Science Foundation grant ISF-2480/20.}}
\author[1]{Yuval Peled\thanks{\href{mailto:yuval.peled@mail.huji.ac.il}{yuval.peled@mail.huji.ac.il}.}}
\author[1]{Orit E. Raz\thanks{\href{mailto:oritraz@mail.huji.ac.il}{oritraz@mail.huji.ac.il}.}}

\affil[1]{Einstein Institute of Mathematics,
 Hebrew University, Jerusalem~91904, Israel}
 \affil[2]{Dept. Math. Sciences, Carnegie Mellon University, Pittsburgh, PA 15213, USA}
	
	\date{}
\begin{document}

	\maketitle

\begin{abstract}
Jord\'an and Tanigawa recently introduced the $d$-dimensional algebraic connectivity $a_d(G)$ of a graph $G$. This is a quantitative measure of the $d$-dimensional rigidity of $G$ which generalizes the well-studied notion of spectral expansion of graphs. We present a new lower bound for $a_d(G)$ defined in terms of the spectral expansion of certain subgraphs of $G$ associated with a partition of its vertices into $d$ parts. In particular, we obtain a new sufficient condition for the rigidity of a graph $G$. 
As a first application, we prove the existence of an infinite family of $k$-regular $d$-rigidity-expander graphs for every $d\ge 2$ and $k\ge 2d+1$. Conjecturally, no such family of $2d$-regular graphs exists. Second, we show that $a_d(K_n)\geq \frac{1}{2}\lfrac{n}{d}$, which we conjecture to be essentially tight. In addition, we study the extremal values $a_d(G)$ attained if $G$ is a minimally $d$-rigid graph.
\end{abstract}

\section{Introduction}

Graph expansion is one of the most influential concepts in modern graph theory, with numerous applications in discrete mathematics and computer science (see ~\cite{hoory2006expander,lubotzky2012expander}). Intuitively speaking, an expander is a ``highly-connected" graph, and a standard way to quantitatively measure the connectivity, or expansion, of a graph uses the spectral gap in its Laplacian matrix. A main theme in the study of expander graphs deals with the construction of sparse expanders. In particular, bounded-degree regular expander graphs have been studied extensively in various areas of mathematics ~\cite{gabber1981explicit,LPS,zigzag,friedman_alon,MSS}. This paper studies a generalization of spectral graph expansion that was recently introduced by Jord\'an and Tanigawa via the theory of graph rigidity~\cite{jordan2022rigidity}.

A $d$-dimensional framework is a pair $(G,p)$ consisting of a graph $G=(V,E)$ and a map $p:V\to \Rea^d$. The framework is called $d$-rigid if every continuous motion of the vertices starting from $p$ that preserves the distance between every two adjacent vertices in $G$, also preserves the distance between \emph{every pair} of vertices; see e.g.~\cite{Connelly:RigiditySurvey, graver1993book} for background on framework rigidity. 

Asimow and Roth showed in \cite{AR1} that if the map $p$ is generic (e.g. if the $d|V|$ coordinates of $p$ are algebraically independent over the rationals), then the framework rigidity of $(G,p)$ does not depend on the map $p$. Moreover, they showed that for a generic $p$, rigidity coincides with the following stronger linear-algebraic notion of infinitesimal rigidity. 

For every  $u,v\in V$ we define $d_{uv}\in \Rea^d$ by
\[
    d_{u v}=\begin{cases}
           \frac{p(u)-p(v)}{\|p(u)-p(v)\|} & \text{ if } p(u)\neq p(v),\\
            0 & \text{ otherwise,}
    \end{cases}
\]
and $\bv_{u,v}:= (1_u-1_v)\otimes d_{u v}\in \Rea^{d|V|}$,
where $\{1_u\}_{u\in V}$ is the standard basis of $\Rea^{|V|}$ and $\otimes$ denotes the Kronecker product.
 Equivalently,
\[
   \bv_{u,v} ^T= \kbordermatrix{
     & &  &  & u & & & & v & & & \\
     & 0 & \ldots & 0 & d_{u v}^T & 0 & \ldots & 0 & d_{v u}^T & 0 & \ldots & 0}.
\]
 The (normalized) \emph{rigidity matrix} $R(G,p)\in \Rea^{d|V|\times |E|}$ is the matrix whose columns are the vectors $\bv_{u,v}$ for all $\{u,v\}\in E$.
We always assume that the image $p(V)$ does not lie on any affine hyperplane in $\Rea^d$. In such a case, it is possible to show (see \cite{AR1}) that $\rank(R(G,p))\leq d|V|-\binom{d+1}{2}$. The framework $(G,p)$ is called \emph{infinitesimally rigid} if this bound is attained, that is, if $\rank(R(G,p))= d|V|-\binom{d+1}{2}$.

A graph $G$ is called \emph{rigid in $\Rea^d$}, or \emph{$d$-rigid}, if it is
infinitesimally rigid with respect to some map $p$ (or, equivalently, if it is infinitesimally rigid for all generic maps~\cite{AR1}).

For $d=1$ and an injective map $p:V\to \Rea^d$, the rigidity matrix $R(G,p)$ is equal to the incidence matrix of $G$, hence both notions of rigidity coincide with graph connectivity. One can extend this analogy and define a higher dimensional version of the graph's Laplacian matrix, that is called the \emph{stiffness matrix} of $(G,p)$, and is defined by
\[
    L(G,p)=R(G,p) R(G,p)^T \in \Rea^{d|V|\times d|V|}.
\] 
We denote by $\lambda_i(A)$ the $i$-th smallest eigenvalue of a symmetric matrix $A$.
Since $\rank(L(G,p))=\rank(R(G,p))\le d|V|-\binom{d+1}{2}$, the kernel of $L(G,p)$ is of dimension at least $\binom{d+1}{2}$. Therefore, $\lambda_{\binom{d+1}{2}+1}(L(G,p))\neq 0$ if and only if $(G,p)$ is infinitesimally rigid. 

In \cite{jordan2022rigidity}, Jord\'an and Tanigawa defined the \emph{$d$-dimensional algebraic connectivity} of $G$, $a_d(G)$, as
\[
a_d(G)=\sup\left\{ \lambda_{\binom{d+1}{2}+1}(L(G,p)) \middle| \, p: V\to \Rea^d\right\}.
\]
For $d=1$, $L(G,p)$ coincides with the graph Laplacian matrix $L(G)$, and $a_1(G)=a(G)$ is the usual algebraic connectivity, or Laplacian spectral gap, of $G$, introduced by Fiedler in \cite{fiedler1973algebraic}. For every $d\ge 1$, $a_d(G) \ge 0$ since $L(G,p)$ is positive semi-definite, and $a_d(G)>0$ if and only if $G$ is $d$-rigid.

The following notion of \emph{rigidity expander graphs} extends the classical notion of (spectral) expander graphs, corresponding to the $d=1$ case:

\begin{definition}\label{def:expander}
Let $d\geq 1$. A family of graphs $\{G_i\}_{i\in\mathbb{N}}$ of increasing size is called a \emph{family of $d$-rigidity expander graphs} if there exists $\epsilon>0$ such that $a_d(G_i)\geq \epsilon$ for all $i\in\mathbb{N}$.
\end{definition}

It is well known that, for every $k\geq 3$, there exist families of $k$-regular ($1$-dimensional) expander graphs (see e.g. \cite{hoory2006expander}).
Our main result is an extension of this fact to general $d$:

\begin{theorem}\label{thm:rigidity_expanders}
Let $d\geq 1$ and $k\ge 2d+1$. Then, there exists an infinite family of 
$k$-regular $d$-rigidity expander graphs.
\end{theorem}

It was conjectured by Jord\'an and Tanigawa that families of $2d$-regular $d$-rigidity expanders do not exist
(see \cite[Conj. 2]{jordan2022rigidity} for the statement in the $d=2$ case, and see \cite[Conj. 6.2]{lew2022d} for the general case), and clearly families of $k$-regular $d$-rigidity expanders do not exist for $k<2d$ since, for $n$ large enough, such graphs have less than $dn-\binom{d+1}{2}$ edges, and are therefore not even $d$-rigid. Therefore, assuming this conjecture, our result is sharp.

Our main tool for the proof of Theorem \ref{thm:rigidity_expanders} is a new lower bound on $a_d(G)$, given in terms of the ($1$-dimensional) algebraic connectivity of certain subgraphs of $G$ associated with a partition of its vertex set into $d$ parts. For convenience, we let $a(G)=\infty$ if $G$ consists of a single vertex.

Let $G=(V,E)$ be a graph, and let $A,B\subset V$ be two disjoint sets. We denote by $G[A]$ the subgraph of $G$ induced on $A$, and by $G(A,B)$ the subgraph of $G$ with vertex set $A\cup B$ and edge set $E(A,B)=\{e\in E:\, |e\cap A|=|e\cap B|=1\}$. In addition, we say that a partition $V=A_1\cup\cdots\cup A_d$ is \emph{non-trivial} if $A_i\neq \emptyset$ for all $i=1,\ldots,d$. 

\begin{theorem}\label{thm:lower_bound_general_d}
Let $d\ge 2$. For every graph $G=(V,E)$ and a non-trivial partition $V=A_1\cup \cdots\cup A_d$ there holds
\[
a_d(G)\geq \min\left(\bigg\{a(G[A_i]) \bigg\}_{1\leq i\leq d} \bigcup\left\{\frac{1}{2}a(G(A_i,A_j))\right\}_{1\leq i<j\leq d}\right).
\]
In particular, if $G[A_i]$ is connected for all $i\in[d]$ and $G(A_i,A_j)$ is connected for all $1\leq i< j\leq d$, then $G$ is $d$-rigid.
\end{theorem}

\begin{remark}\label{rem:2d}
In the $d=2$ dimensional case, the statement in Theorem \ref{thm:lower_bound_general_d} can be slightly improved (by removing the constant $1/2$) to
\[a_2(G)\geq \min\{a(G[A_1]), a(G[A_2]),a(G(A_1,A_2))\},\] for every non-trivial partition $A_1,A_2$  of $V$.
\end{remark}

In the case $d=2$, we can think of Theorem \ref{thm:lower_bound_general_d} as a quantitative version of (a special case of) a theorem of Crapo \cite[Theorem 7]{crapo1990generic}. For $d\geq 3$, Theorem \ref{thm:lower_bound_general_d} seems to give, in addition to a lower bound on $a_d(G)$, a new sufficient condition for $d$-rigidity, which we believe to be of independent interest (this sufficient condition could also be derived from \cite[Theorem 5.5]{lindemann2022combinatorial}).


To derive Theorem \ref{thm:rigidity_expanders} from Theorem \ref{thm:lower_bound_general_d} we consider a balanced partition of the vertex set, and construct each of the $\binom{d+1}{2}$ subgraphs induced by the partition in separate. In the main case $k=2d+1$, our ``building blocks" are ($1$-dimensional) expander graphs with maximum degree $3$ and a large proportion of vertices of degree $2$. Such graphs are constructed by subdividing edges in  classical constructions of $3$-regular expander graphs. In Theorem \ref{thm:subdivided_spectral_gap} below, we hedge the effect of edge subdivision on the algebraic connectivity of the graph.

For another application of Theorem \ref{thm:lower_bound_general_d}, we derive a slight improvement of the previously known lower bound for $a_d(K_n)$ from \cite[Theorem 1.5]{lew2022d}.

\begin{corollary}\label{cor:kn_lower_bound}
Let $d\geq 3$ and $n\geq d+1$. Then
\[
    a_d(K_n)\geq \frac{1}{2}\left\lfloor\frac{n}{d}\right\rfloor.
\]
\end{corollary}

In addition, we establish the following upper bound on $a_d(G)$, generalizing the case $d=2$ proved by Jord\'an and Tanigawa in \cite[Theorem 4.2]{jordan2022rigidity}.

\begin{theorem}\label{thm:upper_bound}
Let $d\geq 2$, and let $G$ be a graph. Then,
\[
    a_d(G)\leq a(G).
\]
\end{theorem}
Theorem~\ref{thm:upper_bound} was proved recently and independently in~\cite{presenza2022upper}. Our proof is different, using the probabilistic method, and we believe it to be of independent interest.

Finally, we study how small and how large can $a_d(G)$ be provided that $G$ is a minimally $d$-rigid graph. A graph $G$ is called \emph{minimally $d$-rigid} if it is $d$-rigid, but $G\setminus e$ is not $d$-rigid for every edge $e\in E$. For $d=1$, these are exactly spanning trees. This question is related to the aforementioned conjecture that no $2d$-regular $d$-rigidity expanders exist (see Conjecture~\ref{conj:percent-ev}).

Among the minimally $d$-rigid graphs, $a_d(G)$ is maximized by a $d$-analog of the star graph. For every $d\geq 1$ and $n\geq d+1$, let $S_{n,d}$ be the graph consisting of a clique of size $d$, and $n-d$ additional vertices, each adjacent to all of the vertices of the clique, and not adjacent to any other vertex. It is easy to check that $S_{n,d}$ is minimally $d$-rigid.
\begin{theorem}\label{thm:upper_bound_minimally_rigid}
For every $d\geq 1$ and $T\neq K_2, K_3$ a minimally $d$-rigid graph there holds
\[
    a_d(T)\leq 1,
\]
and equality holds if $T=S_{n,d}$.
\end{theorem}
This extends a result of Fiedler (see \cite[4.1]{fiedler1973algebraic}, more explicitly stated by Merris in \cite[Cor. 2]{merris1987characteristic}) corresponding to the case $d=1$. Note that 
for $T=K_2$ (which is a minimally $1$-rigid graph), we have $a_1(K_2)=2$, and for $T=K_3$ (which is a minimally $2$-rigid graph), we have $a_2(K_3)=\frac{3}{2}$ (see \cite[Theorem 4.4]{jordan2022rigidity}).

Considering the other extreme, of minimizers of $a_d$ among all $n$-vertex $d$-rigid graphs,  it was shown by Fiedler in \cite{fiedler1973algebraic} that $a(G)\geq a(P_n)=2(1-\cos(\pi/n))$ for every connected graph $G$, where $P_n$ is the $n$-vertex path 
(see \cite{grone1990laplacian} for an explicit statement). We conjecture that a similar situation holds in higher dimensions: in Subsection~\ref{subsec:GeneralizedPathGraphs} we define \emph{generalized path graphs} $P_{n,d}$, which are certain $n$-vertex minimally $d$-rigid graphs, and provide in Proposition~\ref{prop:generalized_path_graph} bounds on their $d$-dimensional algebraic connectivity implying that 
$a_d(P_{n,d})=\Theta_d(1/n^2)$. We conjecture that these graphs are extremal:
\begin{conjecture}\label{conj:minimal_spectral_gap_obtained_by_paths}
Let $G$ be a $d$-rigid graph on $n$ vertices. Then,
\[
    a_d(G)\geq a_d(P_{n,d}).
\]
\end{conjecture}

The paper is organized as follows: In Section \ref{sec:prelims} we present some results about stiffness matrices that are used later. In particular, we recall the definition of the \emph{lower stiffness matrix} $L^{-}(G,p)$ introduced in \cite{lew2022d}. In Section \ref{sec:lower_bound} we prove Theorem \ref{thm:lower_bound_general_d}. In Section \ref{sec:subdivisions} we study the effects of edge subdivisions on the spectral gap of a graph. Section \ref{sec:rigidity_expanders} contains the proof of our main result, Theorem \ref{thm:rigidity_expanders}, showing the existence of $k$-regular $d$-rigidity expanders for $k\geq 2d+1$. In Section \ref{sec:upper_bound} we give a proof of the upper bound $a_d(G)\leq a(G)$ (Theorem \ref{thm:upper_bound}). In Section \ref{sec:minimally_rigid} we study the $d$-dimensional algebraic connectivity of minimally $d$-rigid graphs. We conclude in Section \ref{sec:concluding} with several open problems and directions for further research.

\section{Preliminaries}\label{sec:prelims}

Occasionally, it is simpler to work with the \emph{lower stiffness matrix} of the framework $(G,p)$, defined by
\[
    L^{-}(G,p)=R(G,p)^T R(G,p) \in \Rea^{|E|\times |E|}.
\] 
By standard linear algebra, we have that $\rank(L(G,p))=\rank(L^{-}(G,p))=\rank(R(G,p))$ and that the non-zero eigenvalues  of $L(G,p)$, with multiplicities, coincide with those of $L^{-}(G,p)$. In particular, assuming that $|E|\geq d|V|-\binom{d+1}{2}$, we have
\begin{equation}\label{eqn:lam_k}
    \lambda_{k}(L(G,p))= \lambda_{|E|-d|V|+k}(L^{-}(G,p)),
\end{equation}

for every $k\geq \binom{d+1}{2}+1$. In addition, the entries of $L^{-}(G,p)$ are given explicitly by the following lemma.
\begin{lemma}[{\cite[Lemma 2.1]{lew2022d}}]
\label{lemma:down_laplacian}
Let $(G,p)$ be a $d$-dimensional framework. Then, for every $e,e'\in E(G)$,
\[
    L^{-}(G,p)_{e,e'}= \begin{cases}
                    2 & \text{ if } e=e'=\{u,v\} \text{ and } p(u)\neq p(v),\\
  d_{uv}\cdot d_{uw} & \text{ if } e=\{u,v\},~e'=\{u,w\} \\
                    0 & \text{ otherwise,}
                    \end{cases} 
\]
\end{lemma}
where $d_{uv}\cdot d_{uw}$ denotes the dot product. In the case that $e=\{u,v\}$ and $e'=\{u,w\}$, we denote by
$\theta(e,e')$ the angle between $d_{uv}$ and $d_{uw}$. Hence,  
    $L^{-}(G,p)_{e,e'}=\cos(\theta(e,e'))$ (by convention, $\cos(\theta(e,e'))=0$ if $d_{uv}=0$ or $d_{uw}=0$).

\section{A lower bound on $a_d(G)$}\label{sec:lower_bound}

We turn to the proof of Theorem \ref{thm:lower_bound_general_d}, starting with the following very simple lemma about the eigenvalues of a block diagonal matrix.

For convenience, given a ``$0\times 0$" matrix $M$, we define $\lambda_1(M)=\infty$. 

\begin{lemma}\label{lemma:block_diagonal}
Let $M\in\Rea^{n\times n}$ be a block diagonal matrix, with blocks $M_1,\ldots,M_k$, where $M_i\in \Rea^{n_i\times n_i}$ is symmetric for every $1\leq i\leq k$. Then, for every $1\leq m\leq n$ and $r_1,\ldots, r_k$ satisfying $m=1-k+\sum_{i=1}^k r_i$ there holds
\[
\lambda_m(M)\geq \min \{ \lambda_{r_i}(M_i) : \, 1\leq i\leq k\}.
\]
\end{lemma}
\begin{proof}
Note that the spectrum of $M$ is the union of the spectra of $M_1,\ldots,M_k$. Let $\lambda=\min \{ \lambda_{r_i}(M_i) : \, 1\leq i\leq k\}$. Then, there are at least 
\[
 \sum_{i=1}^k (n_i-r_i+1)= n+k-\sum_{i=1}^k r_i= n-m+1
\]
eigenvalues of $M$ that are greater or equal to $\lambda$. Therefore, $\lambda$ is one of the $m$ smallest eigenvalues of $M$. That is, $\lambda_m(M)\geq \lambda$ as claimed.
\end{proof}
\begin{remark}
    Note that, under the convention $\lambda_1(M_i)=\infty$ for $M_i\in \Rea^{0\times 0}$, Lemma \ref{lemma:block_diagonal} holds 
     also if we allow values $n_i=0$ and $r_i=1$ for one or more $i\in[k]$.
\end{remark}

\begin{proof}[Proof of Theorem \ref{thm:lower_bound_general_d}]
Suppose that $V=[n]$ and denote $m=|E|-d|V|+\binom{d+1}{2}+1$.

If one of the subgraphs $G[A_i]$ for $i\in[d]$ or $G(A_i,A_j)$ for $1\leq i<j\leq d$ is not connected, then the claim holds trivially. Therefore, we will assume that all of these subgraphs are connected. In particular, we have $|E(G[A_i])|\geq |A_i|-1$ for all $i\in[d]$, and $|E(G(A_i,A_j))|\geq |A_i\cup A_j|-1$ for all $1\leq i<j\leq d$.

Let $x_1,\ldots,x_d\in \Rea^{d-1}$ be the vertices of a regular $(d-1)$-dimensional simplex with edge lengths $1$.
For every $c>0$, consider the embedding $p_c: V\to \Rea^d$ that is given by
\[
    p_c(u)=\left(c x_{i},u \right)
\]
for every $i\in[d]$ and $u\in A_i$. 

Let $L^-$ be the entry-wise limit of the matrices $L^{-}(G,p_c)$ as $c\to\infty$. Note that $\lambda_m(L^-(G,p_c))$ converges to $\lambda_{m}(L^{-})$ as $c\to\infty$ by the continuity of eigenvalues. Therefore,
\begin{equation}
\label{eqn:a_d_and_L_minus}
a_d(G) \ge \lim_{c\to\infty}\lambda_{\binom{d+1}{2}+1}(L(G,p_c))
= \lim_{c\to\infty}\lambda_{m}(L^-(G,p_c))
= \lambda_{m}(L^-).
\end{equation}
We now turn to establish a lower bound for $\lambda_{m}(L^-).$

\begin{claim}\label{clm:Lminus}
For every $e,e' \in E$ there holds,
 \begin{align}\label{eq:limit_lap_general_d}
    L^{-}_{e,e'}=\begin{cases}
                    2 & \text{ if }  e=e',\\
                    \sigma(e,e') & \text{ if } |e\cap e'|=1, \, e,e'\in E(G[A_i]) \text{ for some } i,\\
                    1 & \text{ if } |e\cap e'|=1,  \, e,e'\in E(G(A_i,A_j)) \text{ for some } i\neq j,\\
                    1/2 & \text{ if } |e\cap e'|=1,\, e\in E(G(A_i,A_j)),e'\in E(G(A_i,A_k)) \\ & \quad \text{ for some } i\neq j\neq k,\\
                    0 & \text{otherwise},        
    \end{cases}
 \end{align}
where $\sigma(e,e')=\mathrm{sign}((u-v)(u-w))\in\{\pm 1\}$ if $e=\{u,v\}$ and $e'=\{u,w\}$.
\end{claim}
\begin{proof}
The cases $e=e'$ and  $e\cap e'=\emptyset$ follow directly from Lemma \ref{lemma:down_laplacian}. For the other cases, we observe that if $u\in A_i,~v\in A_j$ for some $i\neq j$, then $d_{uv}$ converges to $(x_i-x_j,0)\in\Rea^d$ as $c\to\infty$. Otherwise, if $u,v \in A_i$ then $d_{uv}=\mathrm{sign}(u-v)\cdot 1_d$ for every $c>0$, where $1_d$ is the $d$-th unit vector in $\Rea^d$. We denote $e=\{u,v\}$ and $e'=\{u,w\}$ and derive the claim by a case analysis using the second case in Lemma \ref{lemma:down_laplacian}:
\begin{itemize}
    \item If $u,v,w \in A_i$ then $d_{uv}\cdot d_{uw}=\mathrm{sign}(u-v)\mathrm{sign}(u-w)=\sigma(e,e').$
    \item If $u\in A_i,~ v,w\in A_j$ then $d_{uv}\cdot d_{uw}\to\|x_i-x_j\|^2=1.$
    \item If $u\in A_i,~ v\in A_j,~w\in A_k$ then $d_{uv}\cdot d_{uw}\to(x_i-x_j)\cdot(x_i-x_k)=1/2$ since the angle between two edges in a regular simplex is $\pi/3$.
    \item If $u,v\in A_i,~w\in A_j$ then $d_{uv}\cdot d_{uw}\to \mathrm{sign}(u-v)1_d\cdot(x_i-x_j,0)=0$.
\end{itemize}
\end{proof}

Let $E'=\cup_{1\leq i<j\leq d} E(G(A_i,A_j))$, and $Q$ an $E\times E$ matrix defined by
\[
Q_{e,e'}=\begin{cases}
 |e\cap e'| & \text{ if } e,e'\in E',\\
 0 & \text{ otherwise.}
\end{cases}
\]
Note that $Q$ is positive semi-definite since $Q=N^TN$, where $N\in \Rea^{V\times E}$ is a (signless) incidence matrix, defined by
\[
    N_{v,e}=\begin{cases}
     1 & \text{ if } v\in e\in E',\\
     0 & \text{ otherwise}.
    \end{cases}
\]
Therefore, by Weyl's inequality (see e.g. \cite[Theorem 2.8.1(iii)]{brouwer2011spectra}), it holds that
\begin{equation}\label{eqn:L2M}
\lambda_{m}(L^{-})\ge \lambda_{m}(M),
\end{equation}
where the matrix $M=L^- -\frac 12Q$ is given directly by
 \begin{align*}
    M_{e,e'}=\begin{cases}
                    2 & \text{ if } 
					e=e'\in E(G[A_i]) \text{ for some } i,\\                    
                    \sigma(e,e') & \text{ if } |e\cap e'|=1, \, e,e'\in E(G[A_i]) \text{ for some } i,\\
                                        1 & \text{ if }  e=e'\in E(G(A_i,A_j)) \text{ for some } i\neq j,\\
                    1/2 & \text{ if } |e\cap e'|=1,  \, e,e'\in E(G(A_i,A_j)) \text{ for some } i\neq j,\\
                    0 & \text{otherwise}.    
    \end{cases}
 \end{align*}
We observe that $M$ is a block diagonal matrix, and denote its blocks by $M_i,~i\in [d]$ and $M_{ij},~1\le i<j\le d$. The block $M_i$ is supported on $E(G[A_i])$ and is equal to the ($1$-dimensional) lower stiffness matrix of $G[A_i]$ with respect to any order-preserving embedding $q_i:A_i \to \Rea$. In particular, by letting $r_i=|E(G[A_i])|-|A_i|+2$ we derive from \eqref{eqn:lam_k} that
\begin{equation}\label{eqn:aGAi}
\lambda_{r_i}(M_i) = \lambda_2(L(G[A_i],q_i)) = a(G[A_i]).
\end{equation}

In addition, the block $M_{ij}$ is supported on $E(G(A_i,A_j))$ and is equal to $1/2$ times the ($1$-dimensional) lower stiffness matrix of $G(A_i,A_j)$ with respect to any embedding $q_{ij}:A_i \cup A_j \to \Rea$ satisfying $q_{ij}(u)<q_{ij}(v)$ if $u\in A_i,~v\in A_j$. In particular, by letting $r_{ij}=|E(G(A_i,A_j))|-|A_i\cup A_j|+2$ we find from \eqref{eqn:lam_k} that
\begin{equation}\label{eqn:aGAij}
\lambda_{r_{ij}}(M_{ij}) = \lambda_2\left(\frac 12 L(G(A_i,A_j),q_{ij})\right) = \frac 12 a(G(A_i,A_j)).
\end{equation}

A direct computation yields that
\[
1-\binom{d+1}{2}+\sum_{i=1}^{d}r_i+\sum_{1\le i<j\le d}r_{i,j} = |E| - d|V| + \binom{d+1}{2} + 1=m.
\]
Hence, by Lemma \ref{lemma:block_diagonal}, \eqref{eqn:aGAi} and \eqref{eqn:aGAij}, we find that
\begin{equation}\label{eqn:M}
\lambda_{m}(M) \ge  
\min\left(\bigg\{a(G[A_i]) \bigg\}_{1\leq i\leq d} \bigcup\left\{\frac{1}{2}a(G(A_i,A_j))\right\}_{1\leq i<j\leq d}\right).
\end{equation}
The desired lower bound for $a_d(G)$ is derived by combining \eqref{eqn:a_d_and_L_minus},\eqref{eqn:L2M} and \eqref{eqn:M}.

In particular, since $G[A_i]$ is connected for all $1\leq i\leq d$, and $G(A_i,A_j)$ is connected for all $1\leq i<j\leq d$, then $a_d(G)>0$, and therefore $G$ is $d$-rigid.
\end{proof}

To derive the stronger bound in the case $d=2$ mentioned in Remark \ref{rem:2d}, we note that in this case $L^-$ itself is a block diagonal matrix with $3$ blocks which are the $1$-dimensional lower stiffness matrices $L^-(G[A_1],q_1)$, $L^-(G[A_2],q_2)$ and $L^-(G(A_1,A_2),q_{12})$ (that is, there is no need for the ``correction" term $Q$). Therefore, by the same reasoning we applied to $M$ in the general case, we find that
\[
a_2(G)\ge \lambda_{m}(L^-) \ge \min(\{a_2(G[A_1]),a_2(G[A_2]),a_2(G(A_1,A_2))\}),
\]
where $m=|E|-2|V|+\binom{2+1}{2}+1$.

\begin{remark}
    The criterion for $d$-rigidity given by Theorem \ref{thm:lower_bound_general_d} is minimal in terms of the edge count. Namely, the assumption that all the $\binom {d+1}2$ graphs in the partition are connected implies that there are at least $|A_i|-1$ edges in $G[A_i]$ for $i\in [d]$, and at least $|A_i|+|A_j|-1$ edges in $G(A_i,A_j)$, for  $1\le i<j\le d.$ In total, there needs be at least
    \[
    \sum_{i=1}^{d}(|A_i|-1) + \sum_{1\le i<j\le d}(|A_i|+|A_j|-1) = d|V|-\binom{d+1}{2}
    \]
    edges in $G$ --- precisely the number of edges in a minimally $d$-rigid graph.
\end{remark}

As a consequence of Theorem \ref{thm:lower_bound_general_d}, we obtain a simple proof of Corollary \ref{cor:kn_lower_bound}, giving a lower bound on the $d$-dimensional algebraic connectivity of $K_n$.

\begin{proof}[Proof of Corollary \ref{cor:kn_lower_bound}]
Partition the vertex set $[n]$ into $d$ sets $A_1,\ldots, A_d$ of sizes as equal as possible. Then, each $A_i$ consist of either $\left\lfloor\frac{n}{d}\right\rfloor$ or $\left\lceil\frac{n}{d}\right\rceil$ vertices. 
The induced subgraphs $G[A_i]$ are just complete graphs on $|A_i|$ vertices, and therefore \[a(G[A_i])\geq |A_i| \geq \lfrac{n}{d}\] for all $i\in[d]$. Similarly, the graphs $G(A_i,A_j)$ are complete bipartite graphs, and therefore \[a(G(A_i,A_j))=\min\{|A_i|,|A_j|\}\geq \lfrac{n}{d}.\]
By Theorem \ref{thm:lower_bound_general_d}, we obtain
\[
    a_d(K_n)\geq \frac{1}{2}\lfrac{n}{d}.
\]
\end{proof}

\begin{remark}
For $d=2$, it was shown by Jord\'an and Tanigawa in \cite{jordan2022rigidity}, relying on a result by Zhu (\cite{zhu2013quantitative}), that $a_2(K_n)=n/2$. Dividing the vertex set into two parts of sizes $\lfloor n/2 \rfloor$ and $\lceil n/2 \rceil$ respectively, we obtain a bound of $a_2(K_n)\geq \lfrac{n}{2}$. This gives a simple proof of the sharp lower bound in the case that $n$ is even.
\end{remark}

We conjecture that the lower bound we obtained in Corollary \ref{cor:kn_lower_bound} is almost tight:

\begin{conjecture}\label{conj:complete}
Let $d\geq 3$ and $n\geq d+1$. Then,
\[
a_d(K_n)=\begin{cases}
1 & \text{ if }  d+1\leq n\leq 2d,\\
\frac{n}{2d} & \text{ if } n\geq 2d.
\end{cases}
\]
\end{conjecture}
Note that this is a strong version of Conjecture 6.1 in \cite{lew2022d}.

\section{Expansion under edge subdivisions}\label{sec:subdivisions}
The goal of this section is to prove the following theorem regarding the effect of edge subdivision on the algebraic connectivity of a graph. Let $G=(V,E)$ be a graph without isolated vertices. Given an edge $e$ in  $G$, replacing $e$ with an induced path containing $m\ge 0$ new internal vertices is called a subdivision of $e$ with $m$ vertices. 

\begin{theorem}\label{thm:subdivided_spectral_gap}
Let $G$ be a connected graph with minimum degree at least $2$ and maximum degree $\Delta$, and let $G'$ be obtained from $G$ by a subdivision of each edge of $G$ with at most $m$ vertices. Then,
\[
a(G')\geq \frac{\min\left\{\frac{1}{\Delta}a(G),4\right\}}{2 (m+1)^2}.
\]
\end{theorem}

 Let $D(G)$ be the diagonal matrix with $D(G)_{i,i}=\deg_G(i)$, and $\mathcal{L}(G)=D(G)^{-\frac{1}{2}} L(G) D(G)^{-\frac{1}{2}}$ be the normalized Laplacian of $G$.  
 The effect of edge subdivision on the normalized Laplacian was studied by Xie, Zhang and Comellas in ~\cite{xie2016normalized}. Denote by $s(G)$ the subdivision of $G$, that is, the graph obtained from $G$ subdividing each edge of $G$ with $1$ vertex, thus subdividing each edge into two edges. Furthermore, let $s^k(G)$ be the $k$-th iterated subdivision. That is, $s^k(G)=s(s^{k-1}(G))$ (where  $s^{0}(G)=G$).

\begin{lemma}[{\cite[Lemma 3.1]{xie2016normalized}}]\label{lemma:normalized_subdivision}
If $\lambda\neq 1$ is an eigenvalue of $\mathcal{L}(s(G))$ then $2\lambda(2-\lambda)$ is an eigenvalue of $\mathcal{L}(G)$.
\end{lemma}

In order to relate the spectral gap of the normalized Laplacian to the one of the unnormalized Laplacian, we will use the following result due to Higham and Cheng ~\cite{higham1998modifying}.

\begin{lemma}[{\cite[Theorem 3.2]{higham1998modifying}}]\label{lemma:inertia}
Let $A\in \Rea^{n\times n}$ be a symmetric matrix, and let $X\in \Rea^{n\times m}$, for some $m\leq n$. Then, for every $1\leq i\leq m$,
\[
    \lambda_{i}(X^{T} A X) = \theta_i \mu_i,
\]
where
\[
\lambda_i(A)\leq \mu_i\leq \lambda_{i+n-m}(A)
\]
and
\[
\lambda_1(X^T X) \leq \theta_i\leq \lambda_m(X^T X).
\]
\end{lemma}

\begin{lemma}\label{lemma:normalized_vs_unnormalized}
Let $G$ be a graph on $n$ vertices, with minimum degree $\delta>0$ and maximum degree $\Delta$. Then, for all $2\leq i\leq n$,
\[
  \delta\leq  \frac{\lambda_i(L(G))}{\lambda_i(\mathcal{L(G)})}\leq \Delta.
\]
\end{lemma}
\begin{proof}
The claim follows directly from Lemma \ref{lemma:inertia}, since $\mathcal{L}(G)=D(G)^{-\frac{1}{2}} L(G) D(G)^{-\frac{1}{2}}$, and \[(D(G)^{-\frac{1}{2}})^T D(G)^{-\frac{1}{2}}=D(G)^{-1}\] has minimal eigenvalue $\frac{1}{\Delta}$ and maximal eigenvalue $\frac{1}{\delta}$. 
\end{proof}

\begin{proposition}\label{prop:iterated_subdivision}
Let $G$ be a connected graph with minimum degree at least $2$ and maximum degree $\Delta$. Then,
\[
a(s^k(G))\geq \frac{\min\left\{\frac{2}{\Delta}a(G),8\right\}}{4^k}.
\]
\end{proposition}
\begin{proof}
First, we will show that
\[
    \lambda_2(\mathcal{L}(s^k(G)))\geq \frac{\min\{\lambda_2(\mathcal{L}(G)),4\}}{4^k}.
\]
We argue by induction on $k$. For $k=0$ the claim is trivial. Let $k\geq 1$, and let $\lambda=\lambda_2(\mathcal{L}(s^{k}(G)))$.
If $\lambda\geq (1/4)^{k-1}$, we are done. Otherwise, assume $\lambda<(1/4)^{k-1}$. In particular, $\lambda< 1$. Thus, by Lemma \ref{lemma:normalized_subdivision}, $2\lambda (2-\lambda)$ is an eigenvalue of $\mathcal{L}(s^{k-1}(G))$,
and it is nonzero as $0<\lambda<1$ ($\lambda$ is positive since $G$ is connected).
Therefore, we find that
\[
    \lambda_2(\mathcal{L}(s^{k-1}(G))) \leq 2\lambda(2-\lambda).
\]  
By the induction hypothesis, we obtain
\begin{align*}
\lambda &\geq \frac{\lambda_2(\mathcal{L}(s^{k-1}(G)))}{2(2-\lambda)} \geq  \frac{\lambda_2(\mathcal{L}(s^{k-1}(G)))}{4} \\
&\geq \frac{1}{4}\frac{\min\{\lambda_2(\mathcal{L}(G)),4\}}{4^{k-1}}= \frac{\min\{\lambda_2(\mathcal{L}(G)),4\}}{4^k}.
\end{align*}
Finally, by Lemma \ref{lemma:normalized_vs_unnormalized}, we obtain
\begin{multline*}
a(s^k(G))=\lambda_2(L(s^k(G)))\geq 2 \lambda_2(\mathcal{L}(s^k(G))) \geq 2\frac{\min\{\lambda_2(\mathcal{L}(G)),4\}}{4^k}
\\
\geq \frac{\min\left\{\frac{2}{\Delta}\lambda_2(L(G)),8\right\}}{4^k} = \frac{\min\left\{\frac{2}{\Delta}a(G),8\right\}}{4^k}.
\end{multline*}
\end{proof}

The next lemma establishes that algebraic connectivity is monotone with respect to edge subdivision.

\begin{lemma}\label{lemma:subdivision_second_direction}
Suppose that $G'=(V',E')$ is obtained from $G=(V,E)$ by a subdivision of an edge $e=uv$ of $G$ with $1$ new vertex $w$. Then,
$
    a(G')\leq a(G).
$
\end{lemma}
\begin{proof}
If $G$ is not connected, then $a(G)=a(G')=0$, and in particular $a(G')\leq a(G)$. Therefore, we may assume that $G$ is connected.

Consider the $1$-dimensional rigidity matrices $R(G',p')$ with respect to some embedding $p':V'\to\Rea$ satisfying $p'(u)<p'(w)<p'(v)$, and $R(G,p)$ with respect to the restriction of $p'$ on $V=V'\setminus\{w\}$. Suppose without loss of generality that the last two columns in $R(G',p')$ are $\bv_{uw},\bv_{wv}$, the last column in $R(G,p)$ is $\bv_{uv}$, and that all the remaining columns are equal. Namely, $R(G,p)$ is obtained from $R(G',p')$ by summing up the last two columns into one. In other words,
\[
R(G,p) = R(G',p')X
\]
where $X\in\Rea^{E'\times E}$ is a matrix whose rows are the $|E|$ standard unit row vectors, having the last unit row vector repeated twice.
Therefore,
\[
L^{-}(G,p)= X^T L^{-}(G',p') X.
\]
Using Lemma \ref{lemma:inertia} and $\lambda_1(X^TX) \ge 1$ we find that for every $1\le i\le |E|$
\[
\lambda_i(L^{-}(G,p)) \ge \lambda_i(L^{-}(G',p')).
\]
By letting $i=|E|-|V|+2=|E'|-|V'|+2$ we find using \eqref{eqn:lam_k} that
\[
a(G) = \lambda_2(L(G,p)) = \lambda_i(L^{-}(G,p)) \ge \lambda_i(L^{-}(G',p')) = \lambda_i(L^{-}(G',p')) = a(G'),
\]
as claimed.
\end{proof}

\begin{proof}
[Proof of Theorem~\ref{thm:subdivided_spectral_gap}]
Let $k=\lceil\log_2(m+1)\rceil$. Note that $s^k(G)$ is obtained from $G$ by adding $\sum_{j=1}^k 2^{j-1}= 2^k-1$ vertices on each edge of $G$. Since $2^k-1\geq m$, $s^k(G)$ can be obtained from $G'$ by subdividing some of its edges. Hence, by the monotonicity established in Lemma \ref{lemma:subdivision_second_direction}, we have that
\[
a(G')\geq a(s^k(G)).
\]
Finally, by Proposition \ref{prop:iterated_subdivision}, we obtain
\[
a(G')\geq a(s^k(G))\geq \frac{\min\left\{\frac{2}{\Delta}a(G),8\right\}}{4^k}
\geq \frac{\min\left\{\frac{2}{\Delta}a(G),8\right\}}{4^{\log_2(m+1)+1}}
=\frac{\min\left\{\frac{1}{\Delta}a(G),4\right\}}{2 (m+1)^2}.
\]
\end{proof}

Next, we apply Theorem~\ref{thm:subdivided_spectral_gap} to show the existence of ($1$-dimensional) expander graphs with a desired degree sequence that are used subsequently as building blocks in the construction of $d$-dimensional rigidity expanders in the proof of Theorem~\ref{thm:rigidity_expanders}.  
\begin{corollary}\label{cor:expanders_with_many_degree_2_vertices}
For every $d\geq 1$ there exists $c=c(d)>0$ and an infinite family of  $2dn$-vertex bipartite graphs $(H_n)_{n=3}^{\infty}$ such that $a(H_n)\geq c$ and each part consists of $n$ vertices of degree $3$ and $(d-1)n$ vertices of degree $2$.    
\end{corollary}

\begin{proof}
We start with an infinite family $(\tilde{H}_n)_{n\ge 3}$ of $2n$-vertex $3$-regular bipartite graphs with $a(\tilde H_n)  \geq \epsilon$ for some $\epsilon>0$. For example, random $2n$-vertex $3$-regular bipartite graphs  have this property with high probability and with a nearly optimal algebraic connectivity of $3-2\sqrt{2}-o(1)$ as $n\to\infty$ ~\cite{brito2022spectral}.

To construct $H_n$ we subdivide the $3n$ edges of $\tilde H_n$ with a total number of $2(d-1)n$ new vertices such that (i) every edge is subdivided with an even number of new vertices and (ii) the new vertices are spread (in pairs) as uniformly as possible among the edges. In particular, no edge is subdivided with more than $2\cdot \lceil (d-1)n /(3n)\rceil=2\lceil (d-1)/3 \rceil$ new vertices. Note that $H_n$ is bipartite with $n$ old vertices of degree $3$ and $(d-1)n$ new vertices of degree $2$ in each part, since all the edges were subdivided with an even number of vertices. In addition, by Theorem \ref{thm:subdivided_spectral_gap}, we have that 
\[
a(H_n) \ge  \frac{\epsilon}{6\left(2 \left\lceil\frac{(d-1)}{3}\right\rceil+1\right)^2}.
\]
\end{proof}

\section{Existence of rigidity expander graphs}\label{sec:rigidity_expanders}
We turn to combine the results in the previous sections to establish the existence of a family of $k$-regular $d$-rigidity expanders for every $k\ge 2d+1$.
\begin{proof}
[Proof of Theorem~\ref{thm:rigidity_expanders}]
We will construct, for all $k\geq 2d+1$ and sufficiently large $n$, a $k$-regular graph $G_n$ on $2d^2 n$ vertices satisfying $a(G_n)\geq c(d)/2$ (where $c(d)>0$ is the constant from Corollary \ref{cor:expanders_with_many_degree_2_vertices}). 

We start with the main case $k=2d+1$. The vertex set $V$ of the graph $G=G_n$ is a disjoint union of $d$ sets $V=A_1\cup\cdots\cup A_d,$ where each $A_i,~i\in[d]$, is itself partitioned into $d$ disjoint subsets $A_i = B_{i,1}\cup\cdots\cup B_{i,d}$. All the $d^2$ ``atomic" vertex subsets $B_{i,j}, ~i,j\in[d]$ are of size $2n$. In particular, $|V|=2d^2n$.

We can and will describe the edge set of $G$ by specifying all the $\binom{d+1}2$ subgraphs  $G[A_i],~i\in[d]$ and $G(A_i,A_j),~1\le i<j\le d$. 
The graph $G(A_i,A_j)$, for $1\le i < j \le d$, is constructed such that it is isomorphic to the $4dn$-vertex bipartite graph $H_{2n}$ from Corollary \ref{cor:expanders_with_many_degree_2_vertices}. We do not specify the isomorphism map between $H_{2n}$ and $G(A_i,A_j)$ precisely but we do require that (i) one part of $H_{2n}$ will be mapped to $A_i$ and the other part to $A_j$, and (ii) the $4n$ vertices of degree $3$ in $H_{2n}$ will be mapped to $B_{i,j}$ and $B_{j,i}$.

Similarly, the graph $G[A_i],~i\in[d]$, is isomorphic to the $2dn$-vertex graph $H_n$ from Corollary \ref{cor:expanders_with_many_degree_2_vertices}. We require that in the isomorphism map from $H_n$ to $G[A_i]$, the $2n$ vertices of degree $3$ in $H_n$ are mapped to $B_{i,i}$ (here we do not use the bipartite structure of $H_n$).

Note that Corollary \ref{cor:expanders_with_many_degree_2_vertices} asserts that $a(G[A_i])\ge c(d),~\forall i\in[d],$ and $a(G(A_i,A_j))\ge c(d),~\forall 1\le i<j\le d.$ By Theorem~\ref{thm:lower_bound_general_d}, using the partition $A_1,...,A_d$, we find that
\[a_d(G)\ge \frac{c(d)}{2},
\] (and twice this lower bound for $d=2$ by Remark \ref{rem:2d}).

We conclude by showing that $G$ is $(2d+1)$-regular. Indeed, let $v\in B_{i,j}$ be some vertex of $G$. If $i=j$, then $v$ is contained in $3$  edges from $G[A_i]$, and $2$ edges from each of the $d-1$ graphs $G(A_i,A_{i'}),~\forall i'\ne i$. Here we slightly abuse the notation by letting $G(A_i,A_{i'})=G(A_{i'},A_i)$. Similarly, if $i\ne j$ then $v$ is contained in $3$ edges from $G(A_i,A_j)$, $2$ edges from $G[A_i]$ and $2$ from each of the $d-2$ graphs $G(A_i,A_{i'}),~\forall i'\ne i,j$.

To derive the claim for $k>2d+1$, recall that Dirac famously proved that every $m$-vertex graph with minimum degree at least $m/2$ is Hamiltonian. In addition, if $m$ is even then the edges in the odd places of Dirac's Hamiltonian path form a perfect matching. 
Therefore, if $n$ is sufficiently large then the complement $\bar G$ of the $(2d+1)$-regular $2d^2 n$-vertex graph $G=G_n$ contains $k-(2d+1)$ edge-disjoint perfect matchings. We can add these matchings to $G$, and obtain a $k$-regular graph $G'_n$ which satisfies $a_d(G'_n) \ge a_d(G_n)\geq c(d)/2$ by the monotonicity of $a_d$ with respect to adding edges. 
\end{proof}

\section{An upper bound on $a_d(G)$}\label{sec:upper_bound}

In this section we give a proof of Theorem \ref{thm:upper_bound}, stating that for every graph $G$, $a_d(G)\leq a(G)$. In Section \ref{sec:concluding} below we discuss whether this theorem is tight.

We will need the following simple result about stiffness matrices.
\begin{lemma}[Jord\'an-Tanigawa {\cite[3.2]{jordan2022rigidity}}]
\label{lemma:quadratic_form}
Let $G=(V,E)$ be a graph, and let $p:V\to \Rea^d$ and $x\in \Rea^{d|V|}$. Then
\[
x^T L(G,p) x = \sum_{\{u,v\}\in E} \left\langle x(u)-x(v), d_{uv}\right\rangle^2,
\]
where $x(u)\in \Rea^d$ consists of the $d$ coordinates of $x$ corresponding to the vertex $u$. 
\end{lemma}

We will also need the following lemma from \cite{lew2022d}, that states that when computing $a_d(G)$, it is enough to consider maps $p:V\to \Rea^d$ that are embeddings (that is, injective).

\begin{lemma}[{\cite[Lemma 2.4]{lew2022d}}]\label{lemma:a_d_equivalent}
Let $G=(V,E)$ be a graph, and $d\geq 1$. Then
\[
a_d(G)=\sup\left\{ \lambda_{\binom{d+1}{2}+1}(L(G,p)) \middle| \, p: V\to \Rea^d, \, \text{ $p$ is injective} \right\}.
\]
\end{lemma}

\begin{proof}[Proof of Theorem \ref{thm:upper_bound}]

Let $p:V\to \Rea^d$ be an injective map.  We want to show that $\lambda_{\binom{d+1}{2}+1}(L(G,p))\leq a(G)$.
If $(G,p)$ is not infinitesimally rigid, then $\lambda_{\binom{d+1}{2}+1}(L(G,p))=0\leq a(G)$, as wanted. Therefore, we may assume that $(G,p)$ is infinitesimally rigid. For convenience, assume $V=[n]$. Let $y\in \Rea^n$ be an eigenvector of $L(G)$ corresponding to eigenvalue $\lambda_2(L(G))=a(G)$. Assume $\|y\|=1$.

For $z\in S^{d-1}$, let $x_z\in \Rea^{dn}$ be defined by
\[
    x_z(i)= y_i z\in \Rea^d
\]
for $i\in[n]$. (Intuitively, all vertices move in direction $z$, vertex $i$ moves at speed $y_i$). 
Then, by Lemma \ref{lemma:quadratic_form},
\[
    x_z^T L(G,p) x_z= \sum_{ij\in E} (y_i-y_j)^2 \left\langle z,d_{ij}\right\rangle^2.
\]
Let $z$ be uniform in $S^{d-1}$. Clearly, if $y\in S^{d-1}$ then $\mathbb{E}_z[\left\langle z,y\right\rangle^2]$ does not depend on $y$ since $z$ is spherically symmetric. Therefore,
\[
\mathbb{E}_z[\left\langle z,y\right\rangle^2] = \frac{1}{d} \mathbb{E}_z \left[\sum_{i=1}^{d}\left\langle z,u_i\right\rangle^2\right] = \frac {\mathbb E[\|z\|^2]}d = \frac 1d,
\]
where $u_1,...,u_d$ is some orthonormal basis for $\Rea^d.$

Since $p$ is injective, we have $\|d_{ij}\|=1$ for all $i<j$. That is, $d_{ij}\in S^{d-1}$, and therefore $\mathbb{E}_z[\left\langle z,d_{ij}\right\rangle^2]=1/d$ for all $i<j$. By linearity of expectation, 
\[
\mathbb{E}_z[x_z^T L(G,p) x_z]=\frac 1d \sum_{ij\in E} (y_i-y_j)^2 = a(G)/d.
\]
Similarly,
\[
    x_z^T L(K_n,p) x_z= \sum_{i<j} (y_i-y_j)^2 \left\langle z,d_{ij}\right\rangle^2.
\]
Note that, since $y\perp \textbf{1}$ and $\|y\|=1$, 
we have $\sum_{i<j} (y_i-y_j)^2 = y^T L(K_n)y  = n$. Thus, we obtain

\[
\mathbb{E}_z[x_z^T L(K_n,p) x_z]=\frac 1d \sum_{i<j} (y_i-y_j)^2=  n/d.
\]
Let 
\[
    Z(z)= n x_z^T L(G,p) x_z - a(G) x_z^T L(K_n,p) x_z,
\]
and note that $\mathbb E_z[Z(z)]=0$. In addition, for any  $i,j$ satisfying $y_i\ne y_j$ there holds 
$$\mathbb P(x_z^T L(K_n,p) x_z=0) \le \mathbb P(\langle z,d_{ij}\rangle ^2 =0)=0.$$
Thus, a vector $z\in S^{d-1}$ satisfies with probability $1$ that $x_{z}\notin\ker L(K_n,p)$, and with a positive probability that $Z(z)\leq 0$. In particular, there is some $z_0\in S^{d-1}$ such that $Z(z_0)\leq 0$ and $x_{z_0}\notin\ker L(K_n,p)$. 
Let $x=x_{z_0}$, and $0\ne q\in \Rea^{dn}$ be the orthogonal projection of $x$ on the image of $L(K_n,p)$, so $x-q$ is the orthogonal projection of $x$ on the space of trivial infinitesimal motions. 
Thus, it holds that
\begin{align}
\nonumber q^T L(G,p) q &=~ x^T L(G,p) x \\ 
\label{eq:Z}  &\le~ \frac {a(G)}{n} x^TL(K_n,p)x\\
\nonumber  &= ~\frac {a(G)}{n} q^TL(K_n,p)q \\
\label{eq:n} &\le~ a(G)\|q\|^2. 
\end{align}
The first equality holds as $(G,p)$ is infinitesimally rigid (and so $\ker(L(G,p))=\ker(L(K_n,p))$).
The inequality \eqref{eq:Z} follows from $Z(z_0)\le 0$, and the transition to \eqref{eq:n} follows since the largest eigenvalue of $L(K_n,p)$ is $n$ ~\cite[Lemma 4.3]{jordan2022rigidity}. 

Since $q$ is orthogonal to the $\binom {d+1}2$-dimensional kernel of $L(G,p)$
, we find using the variational characterization of eigenvalues that
\[
    \lambda_{\binom{d+1}{2}+1}(L(G,p)) \leq \frac{q^T L(G,p) q}{\|q\|^2} = a(G).
    \]
Therefore, by Lemma \ref{lemma:a_d_equivalent}, $a_d(G)\leq a(G)$.
\end{proof}

\section{Minimally rigid graphs}\label{sec:minimally_rigid}

In this section we study the extremal values of the $d$-dimensional algebraic connectivity of minimally $d$-rigid graphs.   In Proposition \ref{prop:upper_bound_minimmaly_rigid} we prove the upper bound $a_d(T)\leq 1$ for minimally $d$-rigid graphs, and in 
Proposition \ref{prop:lower_bound_star_graph} we show that the upper bound is attained for ``generalized star" graphs. This establishes the proof of Theorem \ref{thm:upper_bound_minimally_rigid}. The section is concluded with a discussion about generalized path graphs and their algebraic connectivity.

\begin{proposition}\label{prop:upper_bound_minimmaly_rigid}
Let $d\geq 1$, and let $T\neq K_2, K_3$ be a minimally $d$-rigid graph. Then,
\[
    a_d(T)\leq 1.
\]
\end{proposition}

For the proof of Proposition \ref{prop:upper_bound_minimmaly_rigid} we will need the following results.

\begin{lemma}[Jord\'an-Tanigawa {\cite[Lemma 4.5]{jordan2022rigidity}}]\label{lemma:vertex_removal}
Let $d\geq 1$. Let $G=(V,E)$ and $v\in V$. Then,
\[
    a_d(G\setminus v)\geq a_d(G)-1.
\]
\end{lemma}

\begin{theorem}[{\cite[Theorem 1.2]{lew2022d}}]\label{thm:simplex}
Let $d\geq 3$. Then
\[
    a_d(K_{d+1})=1.
\]
\end{theorem}


\begin{proof}[Proof of Proposition \ref{prop:upper_bound_minimmaly_rigid}]
Denote $T=(V,E)$. 
First, if $|V|=d+1$ then $T=K_{d+1}$ and $d\geq 3$ since we assume that $T\neq K_2, K_3$. Therefore, $a_d(T)=1$ by Theorem \ref{thm:simplex}.

Otherwise, we claim that there is a vertex $v$ in $T$ such that $T\setminus v$ is not $d$-rigid. Indeed, $|E|=d|V|-\binom{d+1}{2}$, so the average degree in $T$ is
\[
    \frac{2|E|}{|V|}= 2d- \frac{d(d+1)}{|V|} >d.
\]
Let $v\in V $ be a vertex of degree at least $d+1$ in $T$. Then, 
\[
|E(T\setminus v)| < |E| - d = d|V(T\setminus v)| - \binom{d+1}{2},
\]
hence $T\setminus v$ is not $d$-rigid. In particular, $a_d(T\setminus v)=0$ and by Lemma \ref{lemma:vertex_removal},
\[
    a_d(T) \le a_d(T\setminus v)+1 =1,
\]
as claimed.
\end{proof}

Let $d\geq 1$ and $n\geq d+1$. Let $S_{n,d}$ be the graph on vertex set $[n]$ with edge set
\[
    E(S_{n,d})=\left\{ \{i,j\}:\, i\in[d],\,  j\in[n]\setminus\{i\}\right\}.
\]
It is easy to check that $S_{n,d}$ is minimally $d$-rigid. 

We consider the following mapping of the vertices of $S_{n,d}$ to $\Rea^d$:
Let $e_1,\ldots,e_d\in \Rea^d$ be the standard basis vectors. We define $\px:[n]\to \Rea^d$ by
\[
    \px(i)=\begin{cases}
        e_i & \text{ if } 1\leq i\leq d,\\
        0 & \text{ if } d<i\leq n.
    \end{cases}
\]

\begin{proposition}\label{claim:star_spectrum}
The spectrum of $L(S_{n,d},\px)$ is
\[
\left\{ 0^{\left(\binom{d+1}{2}\right)}, 1^{\left(dn-\binom{d+1}{2}-d\right)}, (n-d/2)^{(d-1)}, n^{(1)} \right\}
\]
(where the superscript $(m)$ indicates multiplicity $m$ of the corresponding eigenvalue).
In particular, $\lambda_{{\binom{d+1}{2}+1}}(L(S_{n,d},\px))= 1$.
\end{proposition}
\begin{proof}
It will be more convenient to work with the lower stiffness matrix $L^{-}(S_{n,d},\px)$. Let $E=E(S_{n,d})$. Since $|E|=d n-\binom{d+1}{2}$, and $L(S_{n,d},\px)$ and $L^{-}(S_{n,d},\px)$ have the same non-zero eigenvalues, the claim is equivalent to showing that the spectrum of $L^{-}(S_{n,d},\px)$ is
\[
\left\{ 1^{\left(dn-\binom{d+1}{2}-d\right)}, (n-d/2)^{(d-1)}, n^{(1)} \right\}
\]
Denote $L^{-}=L^{-}(S_{n,d},\px)$. Let $i,j,k\in [n]$ be distinct. Let $e=\{i,j\}$ and $e'=\{i,k\}$. By Lemma \ref{lemma:down_laplacian}, we have
\[
L^{-}_{e,e'}=
\begin{cases}
 \frac{1}{2} & \text{ if } i,j,k\in [d],\\
 \frac{\sqrt{2}}{2} & \text{ if } i,j\in[d], k\in[n]\setminus[d],\\
1 & \text{ if } i\in[d], j,k\notin[d],\\
0 & \text{ otherwise,}
\end{cases}
\]
and
\[
    L^{-}_{e,e}=2.
\]

From this description of $L^{-}$, the following characterization of the eigenvectors of $L^{-}$ is immediate:
\begin{claim}\label{claim:star_eigenvector_condition}
Let $x\in\Rea^E$. Then $x$ is an eigenvector of $L^{-}$ with eigenvalue $\lambda$ if and only if it satisfies the following conditions:
\begin{enumerate}
    \item For all $1\leq i<j\leq d$, 
    \begin{multline*}
    (\lambda-2)x(\{i,j\}) = \frac{1}{2}\sum_{k\in[d]\setminus\{i,j\}}(x(\{i,k\})+x(\{j,k\})) \\+\frac{\sqrt{2}}{2}\sum_{k\in[n]\setminus[d]}(x(\{i,k\})+x(\{j,k\})),
    \end{multline*}
    \item For all $i\in[d]$, $j\in[n]\setminus[d]$,
    \[
    (\lambda-2)x(\{i,j\}) = \sum_{k\in[n]\setminus([d]\cup\{j\})} x(\{i,k\}) +\frac{\sqrt{2}}{2}\sum_{k\in[d]\setminus\{i\}} x(\{i,k\}).
   \]
\end{enumerate}
\end{claim}

We will find a basis of $\Rea^E$ consisting of eigenvectors of $L^{-}$.

\textbf{Eigenvalue $\bm{n}$:}
Let $\psi\in\Rea^E$ be defined by
\[
    \psi(e)=\begin{cases}
     \sqrt{2} & \text{ if } i,j\in[d],\\
     1 & \text{ if } i\in[d], j\in[n]\setminus[d]
    \end{cases}
\]
for every $e=\{i,j\}\in E$ such that $i<j$. It is easy to check using Claim \ref{claim:star_eigenvector_condition} that $\psi$ is an eigenvector of $L^{-}$ with eigenvalue $n$.

\textbf{Eigenvalue $\bm{n-d/2}$:} 
We divide into two cases, depending on the parity of $d$:

First, assume that $d$ is even. Let $A$ be a subset of $[d]$ of size $d/2$, and let $\bar{A}$ be its complement in $[d]$. Define $\varphi_A\in \Rea^E$ as
\[
    \varphi_A(e)=\begin{cases}
                    \sqrt{2} & \text{ if } e\subset A,\\
                    -\sqrt{2} & \text{ if } e\subset \bar{A},\\
                    1 & \text{ if } e=\{i,j\} \text{ where } i\in A \text{ and } j\in[n]\setminus[d],\\
                    -1 & \text{ if } e=\{i,j\} \text{ where } i\in \bar{A} \text{ and } j\in[n]\setminus[d],\\
                    0 &\text{ otherwise.}
    \end{cases}
\]
It is not hard to check using Claim \ref{claim:star_eigenvector_condition} that $\varphi_A$ is an eigenvector of $L^-$ with eigenvalue $n-d/2$.

We will show that the span of the set $\{\varphi_A:\, A\subset[d],\, |A|=d/2\}$ has dimension at least $d-1$.

For $A\subset [d]$, $|A|=d/2$, let $\tilde{\varphi}_A$ be the restriction of $\varphi_A$ on the edges of the form $\{i,d+1\}$, for $i\in[d]$. It is enough to show that the dimension of the span of $X=\{\tilde{\varphi}_A:\, A\subset[d],\, |A|=d/2\}$ is at least $d-1$. 

Note that $X$ consists exactly of all the vectors in $\Rea^d$ with half of their coordinates equal to $1$ and half of them equal to $-1$.  We will show that $X$ spans all the vectors in $\Rea^d$ whose sum of coordinates is $0$. Indeed, let $e_1,\ldots,e_d$ be the standard basis vectors of $\Rea^d$. Then, since the vectors $e_i-e_j$, $i<j\in[d]$, span the subspace of vectors with sum $0$, it is enough to show that we can write $e_i-e_j$ as a linear combination of vectors in $X$. Indeed, we can choose $u$ to be any vector in $X$ whose $i$-th coordinate is $1$ and whose $j$-th coordinate is $-1$. Let $v$ be obtained by flipping the $i$-th and $j$-th coordinate of $u$. It is clear that $v$ also belongs to $X$. We can write
\[
e_i-e_j=\frac{1}{2}u-\frac{1}{2}v.
\]  
Therefore, the span of $X$ consists of all zero-sum vectors in $\Rea^d$. Hence, its dimension is $d-1$.

So the dimension of $\text{span}(\{\varphi_A:\, A\subset[d],\, |A|=d/2\})$ is also at least $d-1$, meaning there is a set of $d-1$ linearly independent eigenvectors of $L^-$ with eigenvalue $n-d/2$.

Now, assume that $d$ is odd. For $d=1$ there is nothing to prove, so assume $d\ge 3$. Let $k\in[d]$, let $B$ be a subset of $[d]\setminus\{k\}$ of size $(d-1)/2$, and let $\bar{B}$ be its complement in $[d]\setminus\{k\}$. We define $\phi_{k,B}\in \Rea^E$ by
\[
\phi_{k,B}(e)=\begin{cases}
            \sqrt{2} & \text { if } e\subset B,\\
            -\sqrt{2} & \text{ if } e\subset \bar{B},\\
     \frac{\sqrt{2}}{2} & \text{ if } e=\{i,k\} \text{ for } i\in B,\\
 -\frac{\sqrt{2}}{2} & \text{ if } e=\{i,k\} \text{ for } i\in\bar{B},\\
 1 & \text{ if } e=\{i,j\} \text{ for } i\in B,\, j\in [n]\setminus[d],\\
             -1 & \text{ if } e=\{i,j\} \text{ for } i\in \bar{B},\, j\in [n]\setminus[d],\\
           0 & \text{ otherwise.}  
\end{cases}
\]
It is not hard to check using Claim \ref{claim:star_eigenvector_condition} that $\phi_{k,B}$ is an eigenvector of $L^-$ with eigenvalue $n-d/2$.

We will show that the span of the set $\{\phi_{k,B}:\, k\in[d],\,  B\subset[d]\setminus\{k\},\, |B|=(d-1)/2\}$ has dimension at least $d-1$.

For $k\in [d]$ and $B\subset [d]\setminus\{k\}$ with  $|B|=(d-1)/2$, let $\tilde{\phi}_{k,B}$ be the restriction of $\phi_{k,B}$ on the edges of the form $\{i,d+1\}$, for $i\in[d]$. It is enough to show that the dimension of the span of $Y=\{\tilde{\phi}_{k,B}:\, k\in [d], \, B\subset [d]\setminus\{k\}, \,  |B|=(d-1)/2\}$ is at least $d-1$. 

Note that $Y$ consists exactly of all the vectors in $\Rea^d$ with one coordinate equal to $0$, $(d-1)/2$ coordinates equal to $1$ and $(d-1)/2$ coordinates equal to $-1$.
 We will show that $Y$ spans all the vectors in $\Rea^d$ whose sum of coordinates is $0$. This follows from essentially the same argument we used for the case of even $d$:
 
 Let $e_1,\ldots,e_d$ be the standard basis vectors of $\Rea^d$. Then, since the vectors $e_i-e_j$, $i<j\in[d]$, span the subspace of vectors with sum $0$, it is enough to show that we can write $e_i-e_j$ as a linear combination of vectors in $Y$. Indeed, we can choose $u$ to be any vector in $Y$ whose $i$-th coordinate is $1$ and whose $j$-th coordinate is $-1$ (here we use that $d\geq 3$). Let $v$ be obtained by flipping the $i$-th and $j$-th coordinate of $u$. It is clear that $v$ also belongs to $Y$. We can write
\[
e_i-e_j=\frac{1}{2}u-\frac{1}{2}v.
\]  
Therefore, the span of $Y$ consists of all zero-sum vectors in $\Rea^d$. Hence, its dimension is $d-1$. Therefore, there is a set of $d-1$ linearly independent eigenvectors of $L^-$ with eigenvalue $n-d/2$.

\textbf{Eigenvalue $\bm{1}$:} Let $i\in[d]$ and $j,k\in[n]\setminus[d]$. Define $\Phi_{i,j,k}\in \Rea^E$ as
\[
    \Phi_{i,j,k}(e)=\begin{cases}
                        1 & \text{ if } e=\{i,j\},\\
                        -1 & \text{ if } e=\{i,k\},\\
                        0 & \text{ otherwise.}
    \end{cases}
\]
It is easy to check using Claim \ref{claim:star_eigenvector_condition} that $\Phi_{i,j,k}$ is an eigenvector of $L^{-}$ with eigenvalue $1$.

Let $I=\{\Phi_{i,d+1,k}:\, i\in[d], k\in[n]\setminus[d+1]\}$. Since each vector $\Phi_{i,d+1,k}$ in $I$ has an edge that is unique to its support (the edge $\{i,k\}$), $I$ is linearly independent. Note that $|I|=d(n-d-1)$ .

Now, let $i,j\in[d]$. Define $\Psi_{i,j}\in \Rea^E$ as 
\[
    \Psi_{i,j}(e)=\begin{cases}
                    \sqrt{2}(n-d) & \text{ if } e=\{i,j\},\\
                    -1 & \text{ if } e=\{i,k\} \text{ or } e=\{j,k\} \text{ for } k\in[n]\setminus [d],\\
                    0 & \text{ otherwise.}
    \end{cases}
\]
It is not hard to check using Claim \ref{claim:star_eigenvector_condition} that $\Psi_{i,j}$ is an eigenvector of $L^{-}$ with eigenvalue $1$. 
Let
\[
    J=\{ \Psi_{i,j} :\,  i,j\in[d], i<j \}.
\]
We have $|J|=\binom{d}{2}$. Note that $J$ is a linearly independent set, since each vector $\Psi_{i,j}$ in $J$ has an edge in its support that is unique to it (the edge $\{i,j\}$). Moreover, its easy to check that the vectors in $I$ are orthogonal to the vectors in $J$. Therefore, $I\cup J$ is a set of 
\[
    d(n-d-1)+\binom{d}{2}= dn-\binom{d+1}{2}-d
\]
linearly independent eigenvectors of $L^-$ with eigenvalue $1$.
\end{proof}

\begin{proposition}\label{prop:lower_bound_star_graph}
Let $d\geq 1$ and $n\geq d+1$. Then, unless $d=2$ and $n=3$, we have
\[
    a_d(S_{n,d})=1.
\]
\end{proposition}
\begin{proof}
By Proposition \ref{prop:upper_bound_minimmaly_rigid}, we have $a_d(S_{n,d})\leq 1$ (note that $S_{n,d}\neq K_3$ unless $d=2$ and $n=3$).
On the other hand, by Proposition \ref{claim:star_spectrum}, $\lambda_{\binom{d+1}{2}+1}(L(S_{n,d},p^*))=1$, so $a_d(S_{n,d})= 1$.
\end{proof}

It would be interesting to determine whether the graphs $S_{n,d}$ are the only extremal cases in Theorem \ref{thm:upper_bound_minimally_rigid} (for $d=1$ this is a result of Merris, \cite[Cor. 2]{merris1987characteristic}).

\subsection{Generalized path graphs}\label{subsec:GeneralizedPathGraphs}

Let $n\geq d+1$. Let $P_{n,d}$ be the graph on vertex set $[n]$ with edges \[\left\{ \{i,j\} :\,1\leq i<j\leq n,\,  j-i\leq d\right\}.\] Note that, for $d=1$, $P_n=P_{n,1}$ is just the path with $n$ vertices.
It is not hard to check that, for $n\geq d+1$, $P_{n,d}$ is minimally rigid in $\Rea^d$.

As mentioned in the introduction, Fiedler~\cite{fiedler1973algebraic} showed that $a_1(G)\geq a_1(P_n)=2(1-\cos(\pi/n))$ for every connected graph $G$.
For $d>1$, we do not know the exact value of $a_d(P_{n,d})$, but the following result gives us its order of magnitude:

\begin{proposition}\label{prop:generalized_path_graph}
Let $d\geq 2$ and $n\geq d+1$. Then
\[
    1-\cos\left(\frac{\pi}{2} \ufrac{n}{d}^{-1}\right)
    \leq 
    a_d(P_{n,d})
    \leq
    2d-2\sum_{k=1}^d \cos(2k\pi/n).
\]
For $d=2$ we have a slightly better lower bound, $a_2(P_{n,2})\geq 2(1-\cos(\pi/n))$.
\end{proposition}
For large $n$, we have $1-\cos\left(\frac{\pi}{2}\ufrac{n}{d}^{-1}\right) \approx 1-\cos\left(\frac{d\pi}{2n}\right)\approx \frac{d^2 \pi^2}{8n^2}$.
Moreover, $2d-2\sum_{k=1}^d \cos(2k\pi/n)\leq \frac{2\pi^2 d(d+1)(2d+1)}{3 n^2}$.
Therefore, $a_d(P_{n,d})=\Theta_d(1/n^2)$.

\begin{proof}[Proof of Proposition \ref{prop:generalized_path_graph}]
We begin with the lower bound. Let $A_i= \{k\in[n]:\, k\equiv i \mod d\}$ for $i=1,\ldots,d$. For $i\in[d]$, let $n_i=|A_i|$ and, for $1\leq i<j\leq d$, let $n_{ij}=|A_i|+|A_j|$. Note that $n_i\leq \ufrac{n}{d}$ and $n_{ij}\leq 2\ufrac{n}{d}$ (and, for $d=2$, $n_{12}=n$).

It is easy to check that $G[A_i]\cong P_{n_i}$ for all $i\in[d]$, and $G(A_i,A_j)\cong P_{n_{ij}}$ for all $1\leq i<j\leq d$. Since, for $n\leq n'$, $a(P_n)=2(1-\cos(\pi/n))\geq 2(1-\cos(\pi/n'))=a(P_{n'})$, we obtain
\[
    a(G[A_i])\geq 2\left(1-\cos\left(\pi\ufrac{n}{d}^{-1}\right)\right)
\]
for all $i\in[d]$, and
\[
    a(G(A_i,A_j))\geq 2\left(1-\cos\left(\frac{\pi}{2}\ufrac{n}{d}^{-1}\right)\right)
\]
for $1\leq i<j\leq d$. For $d=2$ we can give the slightly better bound $a(G(A_1,A_2))\geq 2(1-\cos(\pi/n))$. 
For $d\geq 3$ we obtain by Theorem \ref{thm:lower_bound_general_d}
\[
    a_d(P_{n,d})\geq 1-\cos\left(\frac{\pi}{2}\ufrac{n}{d}^{-1}\right),
\]
and for $d=2$, we obtain by Remark \ref{rem:2d}
\[
    a_2(P_{n,2})\geq 2(1-\cos(\pi/n)).
\]

Finally, we prove the upper bound: by Theorem \ref{thm:upper_bound} we have $a_d(P_{n,d})\leq a(P_{n,d})$. Let $C_{n,d}$ be the graph on vertex set $[n]$ with edges \[\{\{i,j\}:\, 1\leq i<j\leq n,\, j-i\leq d \text{ or } n-j+i\leq d\}.\]
We can think of $C_{n,d}$ as a generalized cycle (and indeed, $C_{n,1}$ is the cycle graph $C_n$). Note that $P_{n,d}$ is a subgraph of $C_{n,d}$, and therefore $a(P_{n,d})\leq a(C_{n,d})$. The Laplacian matrix of $C_{n,d}$ has entries
\[
    L(C_{n,d})_{i,j}=\begin{cases}
                        2d & \text{ if } i=j,\\
                        -1 & \text{ if } |j-i|\leq d \text{ or } n-|j-i|\leq d,\\
                        0 & \text{ otherwise.}
    \end{cases}
\]
This is a circulant matrix, and therefore its eigenvalues are (see e.g. \cite{gray2006toeplitz})
\[
   \left\{ 2d- 2\sum_{i=1}^d \cos(2\pi m k /n):\, m=0,\ldots,n-1\right\}.
\]
We obtain
\[
a_d(P_{n,d})\leq a(C_{n,d})=2d- 2\sum_{i=1}^d \cos(2\pi k /n).
\]
\end{proof}

\section{Concluding remarks}\label{sec:concluding}
Many fascinating open problems suggest themselves. In this paper, we showed that families of $k$-regular $d$-rigidity expanders exist for $k>2d$, and it is natural to seek for the best possible construction.
\begin{problem}
Let $d \ge 2$ and $k>2d$ be integers. What is 
$$
c_d(k) := \sup_{(G_n)_{n\in\mathbb N}}\liminf_n a_d(G_n),
$$
where $(G_n)_{n\in\mathbb N}$ runs over families of $k$-regular graphs of increasing size?
\end{problem}
The $1$-dimensional case of this problem is perhaps the most important question in the theory of expander graphs. The Alon-Boppana Theorem asserts an upper bound of $c_1(k)\le k-2\sqrt{k-1}$, and constructions attaining this bound are known as (one-sided) Ramanujan graphs~\cite{LPS,MSS}. 
For $d\ge 2$, the proof of Theorem \ref{thm:rigidity_expanders} gives a lower bound for $c_d(k)$ which applies to all $k\ge 2d+1$, and whose rate of decay is in the order of $1/d^2$ as $d\to\infty$. If $k\ge td$ for some $t\geq 3$, one can easily adapt our methods and attain a lower bound for $c_d(k) \ge c_1(t)/2$ that is independent of $d$. That is, by using $t$-regular bipartite Ramanujan graphs in Section~\ref{sec:rigidity_expanders} instead of the subdivided graphs from Corollary \ref{cor:expanders_with_many_degree_2_vertices}. The question whether $c_d(2d+1)\to 0$ as $d\to\infty$ remains open.

It is known that $a(T) = O(1/n)$ if $T$ is a bounded-degree tree ~\cite{kolokolnikov2015maximizing}, and we conjecture that this phenomenon extends to higher dimensions.
\begin{conjecture}~\label{conj:percent-ev}
Fix integers $d,b\ge 1$. Then,
\[
\max_{G_n} a_d(G_n) \to 0~~~\mbox{as }n\to\infty,
\]
where $G_n$ runs over all minimally $d$-rigid $n$-vertex graphs of max-degree $b$.
\end{conjecture}

A  stronger but still plausible conjecture --- that $$\sup_{(G_n,p_n)}\lambda_{d(d+1)+1}(L(G_n,p_n))\to 0~~~\mbox{as }n\to\infty,$$ where 
$(G_n,p_n)$ runs over all $d$-frameworks of minimally $d$-rigid $n$-vertex graphs of maximum degree $b$ --- would imply that no $2d$-regular $d$-rigidity expanders exist, via interlacing of spectra under adding an edge to a graph (see \cite[Theorem 2.3]{lew2022d}).

Regarding the relations between the values $a_d(G)$, for $G$ fixed and $d$ that varies, we propose the following strengthening of Theorem~\ref{thm:upper_bound}:

\begin{conjecture}(Monotonicity)
Let $1\le d'<d$ be integers and $G$ a graph on $n$ vertices. Then,
$a_d(G) \le a_{d'}(G)$.
\end{conjecture}

In addition, the fact that we do not know if the bound in Theorem \ref{thm:upper_bound} is tight raises the following problem:
\begin{problem}
    What is $\sup_G (a_d(G) / a(G))$ over all connected graphs $G$?
\end{problem}
The best lower bound that we currently have for this problem that applies to every $d$ is $1/d$ which is given by the generalized star graph $S_{n,d}$. Indeed,  $a_d(S_{n,d})=1$ and $a(S_{n,d})=d$.
For the special case $d=2$ we obtained by computer calculations $a_2(P_{12,2})\geq 0.667\cdot a(P_{12,2})$ (where $P_{n,d}$ is the generalized path graph). It remains a possibility that Theorem \ref{thm:upper_bound} is tight, and we suspect that generalized paths might be the extremal examples. 


\section*{Acknowledgements}
Part of this research was done while A.L. was a postdoctoral researcher at the Einstein Institute of Mathematics at the Hebrew University.

\bibliographystyle{abbrv}
\bibliography{biblio}

\end{document}